\documentclass[12pt]{amsart}
\usepackage{amssymb,verbatim,amscd,amsmath,graphicx}
\usepackage{graphicx}
\usepackage{caption}
\usepackage{subcaption}
\usepackage{pdfsync}

\numberwithin{equation}{section}
\newtheorem{theorem}{Theorem}[section]
\newtheorem{lemma}[theorem]{Lemma}
\newtheorem{proposition}[theorem]{Proposition}

\theoremstyle{definition}
\newtheorem{definition}[theorem]{Definition}
\newtheorem{def-prop}[theorem]{Definition-Proposition}

\newtheorem{remark}[theorem]{Remark}
\newtheorem{example}[theorem]{Example}

\newtheorem{notation}[theorem]{Notation}

\DeclareMathOperator{\vol}{vol}
\DeclareMathOperator{\ini}{in}

\def\NZQ{\mathbb}               
\def\NN{{\NZQ N}}

\def\aa{\mathbf{a}}

\def\mm{{\mathfrak m}}

\def\qq{{\frak q}}

\def\gr{\operatorname{gr}}

\def\x{\mathbf{X}}

\def\1{{\bf 1}}
\def\0{{\bf 0}}

\def\b{\mathbf{b}}

\begin{document}


\title{Growth of multiplicities of graded families of ideals}

\author{Huy T\`ai H\`a}
\address{Tulane University \\ Department of Mathematics \\
6823 St. Charles Ave. \\ New Orleans, LA 70118, USA}
\email{tha@tulane.edu}
\urladdr{http://www.math.tulane.edu/$\sim$tai/}

\author{Pham An Vinh}
\address{Department of Mathematics \\ University of Missouri \\ Columbia, MO 65211, USA}
\email{vapnnc@mizzou.edu}

\thanks{H\`a is partially supported by the Simons Foundation (grant \#279786).}

\begin{abstract}
Let $(R,\mm)$ be a Noetherian local ring of dimension $d > 0$. Let $I_\bullet = \{I_n\}_{n \in \NN}$ be a graded family of $\mm$-primary ideals in $R$. We examine how far off from a polynomial can the length function $\ell_R(R/I_n)$ be asymptotically. More specifically, we show that there exists a constant $\gamma > 0$ such that for all $n \ge 0$,
$$\ell_R(R/I_{n+1}) - \ell_R(R/I_n) < \gamma n^{d-1}.$$
\end{abstract}

\maketitle


\section{Introduction} \label{sec.intro}

Let $(R,\mm)$ be a Noetherian local ring of dimension $d > 0$. Let $I_\bullet = \{I_n\}_{n \in \NN}$ be a graded family of $\mm$-primary ideals in $R$ (that is, $I_0 = R$ and $I_mI_n \subseteq I_{m+n}$ for all $m,n \in \NN$). The \emph{volume} of $I_\bullet$ is defined to be
$$\vol(I_\bullet) := \limsup_{n \rightarrow \infty} \frac{\ell_R(R/I_n)}{n^d/d\!},$$
where $\ell_R(-)$ denotes the length function. Classically, if $I_n = I^n$, for all $n$, are powers of a fixed $\mm$-primary ideal $I$ then $\vol(I_\bullet) = e(I)$ is the well known Hilbert-Samuel \emph{multiplicity} of $I$.
In recent years, there has been a surge of interest in studying the volume $\vol(I_\bullet)$, and particularly the asymptotic behavior of $\ell_R(R/I_n)$, when $I_\bullet$ is an arbitrary graded family (cf. \cite{C3, C2, C1, C0, F, HPV, JM, KK, LM, M, NU, UV}). Under mild assumptions on the ring $R$, the following statements have been established:
\begin{enumerate}
\item $\vol(I_\bullet)$ is an actual limit, i.e., the limit $\lim_{n \rightarrow \infty} \dfrac{\ell_R(R/I_n)}{n^d/d!}$ exists; and
\item $\vol(I_\bullet)$ is the same as the \emph{asymptotic multiplicity} of $I_\bullet$, i.e.,
$$\vol(I_\bullet) = \lim_{s \rightarrow \infty} \dfrac{e(I_s)}{s^d}.$$
\end{enumerate}

This program of research originated from Okounkov's work \cite{O1,O2}, in which the asymptotic multiplicity of graded families of algebraic objects was interpreted in terms of the volume of certain cones (the \emph{Okounkov body}). This method was later developed more systematically by Lazarsfeld and Musta\c{t}\v{a} \cite{LM} and by Kaveh and Khovanskii \cite{KK} for graded linear series on projective schemes. In particular, statements (1) and (2) were proved in \cite{LM} when $R$ is essentially of finite type over an algebraically closed field $k$ with $R/\mm = k$. In a series of papers \cite{C3, C2, C1, C0}, Cutkosky used a different approach to extend this result to hold for an arbitrary field $k$. Specifically, he showed that statement (1) holds for all graded families $I_\bullet$ of $\mm$-primary ideals in $R$ if and only if the nilradical of the $\mm$-adic completion of $R$ has dimension strictly less than $d$, and statement (2) holds when $R$ is analytically unramifield. As a consequence, it was also deduced (see \cite[Corollary 6.3]{C1}) that the \emph{epsilon multiplicity} of an ideal $I$, $\epsilon(I) := \limsup_{n \rightarrow \infty} \frac{\ell_R(H^0_{\mm}(R/I^n))}{n^d/d!}$, defined by Ulrich and Validashti \cite{UV}, existed as an actual limit.

It is known (cf. \cite{C1, CHST}) that the volume $\vol(I_\bullet)$ in general can be an irrational number. Thus, $\ell_R(R/I_n)$ asymptotically does not behave like a polynomial. Our motivation in this paper is the question of how far off from a polynomial can the function $\ell_R(R/I_n)$ be asymptotically. More precisely, we investigate the growth of the difference function $\ell_R(R/I_{n+1}) - \ell_R(R/I_n)$.

It was shown in \cite[Theorem 4.5]{C3} that when $R$ is a \emph{regular} local ring of dimension $d > 0$ and $I_\bullet = \{I_n\}_{n \in \NN}$ is a graded \emph{filtration} (i.e., $I_{n+1} \subseteq I_n$ for all $n \in \NN$) of $\mm$-primary ideals, there exists a constant $\gamma > 0$ such that $0 \le \ell_R(R/I_{n+1}) - \ell_R(R/I_n) < \gamma n^{d-1}$ for all $n \ge 0$. Our main result extends this to an arbitrary Noetherian local ring.

\begin{theorem}[Theorem \ref{T1}] \label{thm.intro}
Let $(R, \mm)$ be a Noetherian local ring of dimension $d > 0$. Let $I_\bullet = \{I_n\}_{n \in \NN}$ be an arbitrary graded family of $\mm$-primary ideals in $R$. Then there exists a constant $\gamma > 0$ such that for all $n \ge 0$, we have
$$\ell_R(R/I_{n+1}) - \ell_R(R/I_n) < \gamma n^{d-1}.$$
\end{theorem}

Without the assumption that $I_\bullet$ is a filtration, it is no longer true that $0 \le \ell_R(R/I_{n+1}) - \ell_R(R/I_n)$. In fact, \cite[Theorem 4.3]{C3} gave an example where
$$\liminf_{n \rightarrow \infty} \ell_R(R/I_{n+1}) - \ell_R(R/I_n) = -\infty.$$
It was also pointed out in \cite[Theorem 4.6]{C3} that there existed a graded filtration $I_\bullet = \{I_n\}_{n \in \NN}$ for which
$$\dfrac{\ell_R(R/I_{n+1}) - \ell_R(R/I_n)}{n^{d-1}}$$
is bounded but does not have a limit when $n \rightarrow \infty$. We exhibit in Example \ref{example1} that the constant $\gamma$ in Theorem \ref{thm.intro} may depend on the graded family $I_\bullet$. In Example \ref{example} we recall a graded family of ideals given by Cutkosky \cite{C2} to show that when $\dim R = 0$ the statement of Theorem \ref{thm.intro} may fail for $n \gg 0$. Hence, the conclusion of Theorem \ref{thm.intro} is the best of what we can hope for.

To prove Theorem \ref{thm.intro}, we first let $\qq$ be an $\mm$-primary ideal of $R$ generated by a system of parameters $(x_1, \dots, x_d)$, and pass to the associated graded ring $\gr_\qq(R)$. This ring is a graded $R/\qq$-algebra generated by the residues of $x_1, \dots, x_d$ in $\qq/\qq^2$, and thus can be realized as a quotient ring of the polynomial ring $A = R/\qq[X_1, \dots, X_d]$. By using initial ideal theory, we can reduce our problem to the situation when we have a graded family of monomial ideals in $A$. The statement of Theorem \ref{thm.intro} is eventually obtained by approximating this graded family of monomial ideals in $A$ with a family of powers of the irrelevant ideal $\mm_A$ of $A$.

Our paper is outlined as follows. In the next section, we collect notations and terminology used in the paper. In Section \ref{sec.gr}, we discuss how to pass the problem to the associated graded ring of an $\mm$-primary ideal $\qq$. Section \ref{sec.poly} is devoted to graded families of ideals in a polynomial ring. In this section, we shall obtain a bound for the growth of the length function when the graded family consists of powers of the irrelevant ideal. Our main result, Theorem \ref{thm.intro} is proved in Section \ref{sec.main}.

\vspace{0.5em}
\noindent{\bf Acknowledgement.} We would like to thank Dale Cutkosky for suggesting the problem to us, and for many helpful discussions and invaluable comments.


\section{Notations and Terminology} \label{sec.notation}

We follow standard notations and terminology of \cite{E, H, SZ}. Throughout the paper, $\NN$ will be the set of non-negative integers, $(R,\mm)$ will denote a Noetherian local ring of dimension $d > 0$, and $\ell_R(-)$ will denote the length function. We shall now recall the notion of graded families of ideals, which is the main object of our study.

\begin{definition} A family $I_\bullet = \{I_n\}_{n \in \NN}$ of ideals in $R$ is called a \emph{graded} family if $I_0 = R$ and for all $m,n \in \NN$, we have $I_m I_n \subseteq I_{m+n}$. A graded family of ideals $\{I_n\}_{n \in \NN}$ is called a \emph{filtration} if $I_{n+1} \subseteq I_n$ for all $n \in \NN$.
\end{definition}

Our method in proving the main result is to pass to the associated graded ring, so we shall recall this notion and the initial ideal theory for associated graded rings.

\begin{definition} Let $I_\bullet = \{I_n\}_{n \in \NN}$ be a graded filtration of ideals in $R$.
\begin{enumerate}
\item Let $M$ be an $R$-module. The \emph{associated graded ring} of $M$ with respect to $I_\bullet$ is defined to be
$$\gr_{I_\bullet}(M) := \bigoplus_{n \in \NN} I_nM/I_{n+1}M.$$

\item Let $f \not= 0$ be an element in $R$, and assume that $\bigcap_{n \in \NN} I_n = 0$. The \emph{initial form} of $f$ with respect to $I_\bullet$ is defined to be
    $$\ini_{I_\bullet}(f) := f+ I_{s+1} \in \gr_{I_\bullet}(R),$$
    where $s \ge 0$ is the integer such that $f \in I_s \setminus I_{s+1}$.
\item Let $J \subseteq R$ be an ideal in $R$, and assume that $\bigcap_{n \in \NN} I_n = 0$. The \emph{initial ideal} of $J$ with respect to $I_\bullet$ is defined to be
$$\ini_{I_\bullet}(J) := \big(\ini_{I_\bullet}(f) ~|~ f \in J\big).$$
\end{enumerate}
\end{definition}

\begin{remark} When $I_\bullet = \{I^n\}_{n \in \NN}$ is given by powers of a fixed ideal $I$ in $R$, we shall use $\gr_I(M)$, $\ini_I(f)$ and $\ini_I(J)$ for $\gr_{I_\bullet}(M)$, $\ini_{I_\bullet}(f)$ and $\ini_{I_\bullet}(J)$, respectively.
\end{remark}

Suppose now that $A = R[X_1, \dots, X_d]$ is a polynomial ring over $R$. Let $\mm_A = (X_1, \dots, X_d)$ be the \emph{irrelevant ideal} of $A$. Let $\prec$ be the degree lexicographic monomial ordering on $A$. For $\aa = (a_1, \dots, a_d) \in \mathbb{N}^d$, let $|\aa| = \sum_{i=1}^d a_i$, and we shall write $\x^\aa$ for the monomial $X_1^{a_1} \dots X_d^{a_d}$ in $A$. We shall also denote by $\aa+1$ the smallest element $\aa' \in \NN^d$ such that $\aa'$ is larger than $\aa$ with respect to $\prec$ (i.e., $\x^{\aa'}$ is the immediate succeeding monomial after $\x^\aa$ in the total ordering of the monomials of $A$).

For $\aa \in \NN^d$, let $A_{\succeq \aa}$ be the ideal of $A$ generated by monomials $\{\x^\b ~|~ \aa \preceq \b\}$. It is easy to see that the degree lexicographic monomial ordering $\prec$ on $A$ gives a graded filtration $A_\bullet = \{A_{\succeq \aa}\}_{\aa \in \NN^d}$ of ideals in $A$.

\begin{notation} We shall denote by $\gr_\prec(A)$ the associated graded ring of $A$ with respect to the filtration $A_\bullet = \{A_{\succeq \aa}\}_{\aa \in \NN^d}$, i.e.,
$$\gr_{\prec}(A) = \bigoplus_{\aa \in \mathbb{N}^d} A_{\succeq \aa}/A_{\succeq \aa +1}.$$
For an $A$-module $M$, we shall also define $\gr_{\prec}(M)$ to be
$$\gr_{\prec}(M) := \bigoplus_{\aa \in \NN^d} A_{\succeq \aa} M/A_{\succeq \aa +1}M.$$
\end{notation}

\begin{remark}
Let $\0$ be the zero vector in $\NN^d$. We can identify $A_{\succeq \0}/A_{\succeq \0 + 1}$ with $R$. Note that each graded component $A_{\succeq \aa}/A_{\succeq \aa +1}$ of $\gr_{\prec}(A)$ is a free $R$-module generated by the residue of $\x^\aa$. Thus, there is a canonical map
$$
\phi: \gr_{\prec}(A) = \bigoplus_{\aa \in \mathbb{N}^d} R(\x^\aa + A_{\succeq \aa +1}) \longrightarrow \bigoplus_{\aa \in \mathbb{N}^d} R \x^\aa = A.
$$
It can be seen that $\phi$ is both a ring isomorphism and an $R$-module isomorphism.
\end{remark}

\begin{definition} Let $J \subseteq A$ be an ideal. We define the \emph{leading ideal} of $J$ in $\gr_\prec(A)$ to be
$$
\mathrm{ld}(J) = \bigoplus_{\aa \in \mathbb{N}^d} \dfrac{A_{\succeq \aa} \cap (A_{\succeq \aa +1} + J)}{A_{\succeq \aa +1}},
$$
and the \emph{initial ideal} of $J$ in $A$ with respect to $\prec$ to be
$$
\ini_{\prec}(J) = \phi(\mathrm{ld}(J)).
$$
The ideal $J \subseteq A$ is called a \emph{monomial ideal} if $J = \bigoplus_{\aa \in \mathbb{N}^d} (J \cap R\x^{\aa})$.
\end{definition}

\begin{remark} Note that in general, the initial ideal $\ini_{\prec}(J)$, as we have defined, is not the same as the ideal generated by \emph{leading monomials} of elements in $A$ (with respect to $\prec$). This is because the ring $R = A_\0$ is not necessarily a field.
\end{remark}


\section{Lengths in Associated Graded Rings} \label{sec.gr}

In this section, we shall see how to pass the length function from $R$-modules to that over the associated graded ring of $R$ with respect to an ideal. Recall that if $I \subseteq R$ is an ideal such that $\bigcap_{n \in \NN} I^n = 0$, then for any ideal $J \subseteq R$, the initial ideal of $J$ with respect to $I$ is defined to be
$$\ini_I(J) = (\ini_I(f) ~|~ f \in J),$$
where $\ini_I(f) = f + I^{s+1} \in \gr_I(R)$ with $s \ge 0$ being the integer such that $f \in I^s \setminus I^{s+1}$. From the definition, it can be seen that
$$\ini_I(J) = \bigoplus_{n \in \NN} \frac{I^n \cap (I^{n+1} + J)}{I^{n+1}}.$$

\begin{lemma} \label{L1}
Let $I$ and $J$ be ideals in $R$. Assume that $\bigcap_{n \in \NN} I^n = 0$. Then
$$
\gr_I(R/J) \cong \gr_I(R)/\ini_I(J).
$$
\end{lemma}

\begin{proof}
We shall first compute the $n$th graded component of $\gr_I(R/J)$. We have
\begin{align*}
\frac{I^n(R/J)}{I^{n+1}(R/J)} & = \frac{I^n + J}{J}\Big/\frac{I^{n+1} + J}{J} \cong \frac{I^n + J}{I^{n+1} + J} \\
& \cong \frac{I^n}{I^n \cap (I^{n+1} + J)} \cong \frac{I^n/I^{n+1}}{I^n \cap (I^{n+1} + J)/I^{n+1}}.
\end{align*}
This implies that
$$
\gr_I(R/J) = \bigoplus_{n \in \NN} \frac{I^n(R/J)}{I^{n+1}(R/J)} \cong \frac{\bigoplus_{n \in \NN} I^n/I^{n+1}}{\bigoplus_{n \in \NN} I^n \cap (I^{n+1} + J)/I^{n+1}} = \frac{\gr_I(R)}{\ini_I(J)}.
$$
\end{proof}

\begin{proposition} \label{C1}
Let $\qq$ and $I$ be $\mm$-primary ideals. Then
$$
\ell_R(R/I) = \ell_{R/\qq}(\gr_{\qq}(R)/\ini_{\qq}(I)).
$$
\end{proposition}

\begin{proof}
Observe that since $\qq$ and $I$ are $\mm$-primary ideals, there exists an integer $N$ so that $\qq^N$ is contained in $I$. Thus, $\qq^N(R/I) = 0$. Therefore, we have
\begin{align*}
\ell_R(R/I) & = \sum_{n \in \NN} \ell_R\Big(\frac{\qq^n(R/I)}{\qq^{n+1}(R/I)}\Big) = \sum_{n \in \NN} \ell_{R/\qq}\Big(\frac{\qq^n(R/I)}{\qq^{n+1}(R/I)}\Big) = \ell_{R/\qq}(\gr_{\qq}(R/I)),
\end{align*}
where the sums are finite sums.
The statement now follows from Lemma \ref{L1}.
\end{proof}


\section{Lengths in Polynomial Rings} \label{sec.poly}

This section provides important tools to prove our main result, Theorem \ref{thm.intro}. We shall see how the length function of a graded family of ideals behaves when the family consists of monomial ideals in a polynomial ring. Throughout this section, we shall make an additional assumption on the local ring $(R,\mm)$ that $\dim R = 0$, i.e., $R$ is an Artinian ring. Define $\ell = \ell_R(R) < \infty$.

Let $A = R[X_1, ..., X_d]$ be a polynomial ring over $R$ and let $\mm_A = (X_1, ..., X_d)$ be its irrelevant ideal. Recall that $\prec$ denotes the degree lexicographic monomial ordering on $A$, and for $\aa = (a_1, \dots, a_d) \in \mathbb{N}^d$, we denote by $\aa + 1$ the smallest $\aa' \in \NN^d$ such that $\aa'$ is larger than $\aa$ with respect to $\prec$. Recall also that $A_{\succeq \aa}$ is the $A$-ideal generated by monomials $\{\x^\b ~|~ \aa \preceq \b\}$, $\gr_{\prec}(M)$ is the associated graded ring of an $A$-module $M$ with respect to the filtration $A_\bullet = \{A_{\succeq \aa}\}_{\aa \in \NN^d}$, and
$$
\phi: \gr_{\prec}(A) = \bigoplus_{\aa \in \mathbb{N}^d} R(\x^{\aa} + A_{\succeq \aa+1}) \longrightarrow \bigoplus_{\aa \in \mathbb{N}^d} R \x^{\aa} = A
$$
is the canonical map which is both a ring isomorphism and an $R$-module isomorphism.

We shall now collect some facts about initial ideals and associated graded rings associated to ideals in $A$.

\begin{lemma} \label{L3}
Let $J \subseteq A$ be an ideal such that $\mm_A \subseteq \sqrt{J}$.
\begin{enumerate}
\item[(a)] We have $\ini_{\prec}(J) = \bigoplus_{\aa \in \mathbb{N}^d} R_{\aa}(J)\x^{\aa}$ is a monomial ideal, where
    $$
    R_{\aa}(J) = \{c \in R \mid c \x^{\aa} \in \ini_{\prec}(J)\}.
    $$
    Moreover, if $J$ is a monomial ideal then $\ini_{\prec}(J) = J$.
\item[(b)] There are canonical ring isomorphisms as well as $R$-module isomorphisms
$$
\gr_{\prec}(A/J) \cong \gr_{\prec}(A)/\mathrm{ld}(J) \cong A/\ini_{\prec}(J).
$$
\item[(c)] We have
$$
\ell_R(A/J) = \ell_R(A/\ini_{\prec}(J)) = \sum_{\aa \in \mathbb{N}^d} \big[\ell - \ell_R(R_{\aa}(J)) \big] < \infty.
$$
\end{enumerate}
\end{lemma}

\begin{proof}
(a) By definition, $\text{ld}(J)$ is a homogeneous ideal in $\gr_{\prec}(A)$ and
$$
\text{ld}(J) = \bigoplus_{\aa \in \mathbb{N}^d} \text{ld}(J) \cap (A_{\aa}/A_{\aa+1}).
$$
Passing this equality to $A$ via $\phi$ we get
$$
\ini_{\prec}(J) = \bigoplus_{\aa \in \mathbb{N}^d} \ini_{\prec}(J) \cap R\x^{\aa}.
$$
Thus, $\ini_{\prec}(J)$ is a monomial ideal.

Suppose now that $J$ is a monomial ideal, i.e.,
$$
J = \bigoplus_{\aa \in \mathbb{N}^d} (J \cap R\x^{\aa}).
$$
For each $\aa \in \mathbb{N}^d$ and for every element $c \in R$, we have that
\begin{align*}
c \x^{\aa} \in J \cap R\x^{\aa} & \Leftrightarrow c \x^{\aa} + A_{\succeq \aa+1} \in \frac{A_{\succeq \aa} \cap (A_{\succeq \aa+1} + J)}{A_{\succeq \aa+1}} \\
& \Leftrightarrow c \x^{\aa} = \phi(c \x^{\aa} + A_{\succeq \aa+1}) \in \ini_{\prec}(J) \cap R\x^{\aa}.
\end{align*}
It follows that $J \cap R\x^{\aa} = \ini_{\prec}(J) \cap R\x^{\aa}$ for all $\aa \in \mathbb{N}^d$, and hence, $J = \ini_{\prec}(J)$.

(b) The first isomorphism can be proved by carrying out the same computation as in the proof of Lemma \ref{L1}, while the second isomorphism follows from definition.

(c) Since $\mm_A \subset \sqrt{J}$, there exists some $N \in \NN$ such that $\mm_A^N \subseteq J$. Thus, for $\aa \in \mathbb{N}^d$ such that $|\aa| \ge N$, we have $A_{\succeq \aa} \subseteq \mm_A^N \subseteq J$, and hence, $A_{\succeq \aa}(A/J) = 0$. This yields that
\begin{align*}
\ell_R(A/J) = \sum_{\aa \in \mathbb{N}^d} \ell_R\Big(\frac{A_{\succeq \aa}(A/J)}{A_{\succeq \aa+1}(A/J)}\Big) = \ell_R(\gr_{\prec}(A/J)) = \ell_R(A/\ini_{\prec}(J)),
\end{align*}
where the sum is a finite sum, and the last equality follows from part (b).

Observe that for $\aa \in \mathbb{N}^d$, we have a surjective $R$-module homomorphism
$$
\frac{A_{\succeq \aa}}{A_{\succeq \aa+1}} \twoheadrightarrow \frac{A_{\succeq \aa}+J}{A_{\succeq \aa+1}+J} \cong \frac{A_{\succeq \aa}(A/J)}{A_{\succeq \aa+1}(A/J)}, \text{ where } f + A_{\succeq \aa+1} \mapsto f + A_{\succeq \aa+1} + J.
$$
It follows that
$$
\ell_R\Big(\frac{A_{\succeq \aa}(A/J)}{A_{\succeq \aa+1}(A/J)}\Big) \le \ell_R\Big(\frac{A_{\succeq \aa}}{A_{\succeq \aa+1}}\Big) = \ell < \infty.
$$
Therefore, since the sum $\ell_R(A/J) = \sum_{\aa \in \mathbb{N}^d} \ell_R\big(\frac{A_{\succeq \aa}(A/J)}{A_{\succeq \aa+1}(A/J)}\big)$ is a finite sum, we conclude that $\ell_R(A/J) = \ell_R(A/\ini_{\prec}(J)) < \infty$.

Finally, for the remaining equality, it can be seen that
\begin{align*}
\ell_R(A/\ini_{\prec}(J)) = \ell_R\Big(\frac{\bigoplus_{\aa \in \mathbb{N}^d} R\x^{\aa}}{\bigoplus_{\aa \in \mathbb{N}^d} R_{\aa}(J)\x^{\aa}}\Big) = \ell_R\Big(\bigoplus_{\aa \in \mathbb{N}^d} \frac{R\x^{\aa}}{R_{\aa}(J)\x^{\aa}}\Big) = \sum_{\aa \in \mathbb{N}^d} \big[\ell - \ell_R(R_{\aa}(J))\big].
\end{align*}
The lemma is proved.
\end{proof}

The following is key for the proof of our main result, Theorem \ref{thm.intro}. It is also interesting on its own.

\begin{proposition} \label{L2}
Let $J$ be a monomial ideal of $A$ such that $\mm_A^N \subseteq J$ for some $N \in \NN$, and let $r \in \NN$ be a positive integer. Then there exists a constant $\gamma > 0$, which does not depend on $N$ and $r$, such that
$$
\ell_R(A/\mm_A^rJ) - \ell_R(A/J) \leq \ell^2 \cdot r \cdot{(d-1) + (N+r-1) \choose N+r-1} < \gamma \cdot r \cdot (N+r)^{d-1}.
$$
\end{proposition}

\begin{proof}
It follows from Lemma \ref{L3}.(c) that
\begin{align} \label{E4}
\ell_R(A/\mm_A^rJ) - \ell_R(A/J) = \sum_{\aa \in \mathbb{N}^d} \big[\ell_R(R_{\aa}(J)) - \ell_R(R_{\aa}(\mm_A^rJ))\big].
\end{align}

Let $\pi: \mathbb{N}^d \twoheadrightarrow \mathbb{N}^{d-1}$ be the projection on the first $d-1$ coordinates. Set
$$
\mathcal{P} = \{\aa \in \mathbb{N}^d \mid \ell_R(R_{\aa}(J)) - \ell_R(R_{\aa}(\mm_A^rJ)) > 0\}
$$
and
$$
\mathcal{Q} = \{\b \in \mathbb{N}^{d-1} ~\big|~ |b| < N + r\}.
$$
Since $\mm_A^N \subseteq J$, we have $\mm_A^{N+r} \subseteq \mm_A^rJ$. If $\aa \in \mathcal{P}$ then $R_{\aa}(\mm_A^rJ) \subsetneq R$, and so $\x^{\aa} \notin \mm_A^rJ$. This implies that $\x^{\aa} \notin \mm_A^{N+r}$, and hence, $|\pi(\aa)| \leq |\aa| < N+r$. Thus,
\begin{align} \label{E5}
\pi(\mathcal{P}) \subset \mathcal{Q}.
\end{align}

Fix $\b \in \pi(\mathcal{P})$ and consider $\aa^0 = (\b,a_d^0) \in \mathcal{P}$, in which $a_d^0$ is chosen to be the smallest possible. Then it follows from the definition of $\mathcal{P}$ that
$$
\ell_R(R_{\aa^0}(\mm_A^rJ)) \geq 0 \text{ and } \ell_R(R_{\aa^0}(J)) \geq 1.
$$
Since $J$ is a monomial ideal, by part Lemma \ref{L3}.(a), we have $\ini_{\prec}(J) = J$. Therefore,
$$
R_{\aa^0}(J) X_d^r \x^{\aa^0} \subset \mm_A^r \ini_{\prec}(J) = \mm_A^rJ.
$$
Hence,
$$
R_{\aa^0}(\mm_A^rJ) \subsetneq R_{\aa^0}(J) \subset R_{(\b, a_d^0+r)}(\mm_A^rJ) \subset R_{(\b,a_d^0+r)}(J).
$$

Observe that if $(\b,a) \notin \mathcal{P}$ for all $a \geq a_d^0+r$ then
$$
\#\pi^{-1}(\b) \leq \#\{a_d^0, \dots, a_d^0+r-1\} = r.
$$
Otherwise, let $a_d^1 \ge a_d^0 +r$ be the smallest integer such that $\aa^1 = (\b,a_d^1) \in \mathcal{P}$. It can be seen that
$$
\ell_R(R_{\aa^1}(\mm_A^rJ)) \geq \ell_R(R_{(\b, a_d^0+r)}(\mm_A^rJ)) \geq \ell_R(R_{\aa^0}(J)) \geq 1
$$
and
$$
\ell_R(R_{\aa^1}(J)) \geq \ell_R(R_{\aa^1}(\mm_A^rJ)) + 1 \geq 2.
$$
Similar to our arguments for $\aa^0$, we get
$$
R_{\aa^1}(\mm_A^rJ) \subsetneq R_{\aa^1}(J) \subset R_{(\b, a_d^1+r)}(\mm_A^rJ) \subset R_{(\b,a_d^1+r)}(J).
$$

Observe again that if $(\b,a) \notin \mathcal{P}$ for all $a \geq a^1_d + r$ then
$$
\#\pi^{-1}(\b) \leq \#\{a_d^0, \dots, a_d^0+r-1\} \cup \{a_d^1, \dots, a_d^1+r-1\} = 2r.
$$
Otherwise, let $a_d^2 \geq a_d^1+r$ be the smallest integer such that $\aa^2 = (\b,a_d^2) \in \mathcal{P}$. By a similar argument as before, we have
$$
\ell_R(R_{\aa^2}(\mm_A^rJ)) \geq \ell_R(R_{(\b, a_d^1+r)}(\mm_A^rJ)) \geq \ell_R(R_{\aa^1}(J)) \geq 2
$$
and
$$
\ell_R(R_{\aa^2}(J)) \geq \ell_R(R_{\aa^2}(\mm_A^rJ)) + 1 \geq 3.
$$

Continuing in this fashion, we obtain a strictly increasing sequence $\aa^0 < \aa^1 < \dots$, where
$$
\ell_R(R_{\aa^i}(J)) \geq i+1.
$$
Since $i \leq \ell_R(R_{\aa^i}(\mm_A^rJ)) \leq \ell-1$, this sequence must terminates after at most $\aa^{\ell-1}$. This yields that
\begin{align} \label{E6}
\#\pi^{-1}(\b) \leq \#\cup_{i = 0}^{\ell-1} \{a_d^i, \dots, a_d^i+r-1\} = \ell \cdot r.
\end{align}

It now follows from (\ref{E4}), (\ref{E5}) and (\ref{E6}) that
\begin{align*}
\ell_R(A/\mm_A^rJ) - \ell_R(A/J) & \leq \ell \cdot \#\mathcal{P} \leq \ell \cdot (\ell \cdot r) \cdot \#\mathcal{Q} \\
& = \ell^2 \cdot r \cdot{(d - 1) + (N+r-1) \choose N+r-1}.
\end{align*}
Finally, since ${(d - 1) + (N+r-1) \choose N+r-1}$ is a polynomial of degree $d-1$ in $N+r-1$, the last inequality must hold for some sufficiently large constant $\gamma > 0$.
\end{proof}


\section{Growth of Lengths of Graded Families} \label{sec.main}

In this section, we prove our main result of the paper, Theorem \ref{thm.intro}. Our arguments go along the following line. Consider an $\mm$-primary ideal $\qq$ that is generated by a system of parameters $(x_1, \dots, x_d)$. Then $\gr_\qq(R)$ is a graded $R/\qq$-algebra generated by the residues of $x_1, \dots, x_d$ in $\qq/\qq^2$. Thus, there is a surjective homomorphism $A = R/\qq[X_1, \dots, X_d] \twoheadrightarrow \gr_\qq(R)$ from a polynomial ring over $R/\qq$ to $\gr_\qq(R)$, and we can realize $\gr_\qq(R)$ as a quotient ring $A/\mathfrak{a}$ of $A$. Note that $R/\qq$ is of dimension 0. Therefore, by passing the length function to the associated graded ring $\gr_\qq(R)$, we are reduced to the situation in Section \ref{sec.poly}. Our result will then follow from Proposition \ref{L2}.

\begin{theorem} \label{T1}
Let $(R,\mm)$ be a Noetherian local ring of dimension $d > 0$, and let $I_\bullet = \{I_n\}_{n \in \NN}$ be an arbitrary graded family of $\mm$-primary ideals of $R$. Then there exists a constant $\gamma > 0$ such that for all $n \ge 0$, we have
$$
\ell_R(R/I_{n+1}) - \ell_R(R/I_n) < \gamma \cdot n^{d-1}.
$$
\end{theorem}

\begin{proof} As we have outlined above, let $\qq$ be an $\mm$-primary ideal of $R$ generated by a system of parameters $(x_1, \dots, x_d)$, and let $A = R/\qq[X_1, \dots, X_d]$ be the polynomial ring such that $\gr_\qq(R) \cong A/\mathfrak{a}$, where $\mathfrak{a} \subset A$ is a homogeneous ideal. Let $\mm_A = (X_1, \dots, X_d)$ be the irrelevant ideal of $A$. Note that $\mm_A$ is the maximal ideal of $A$ only when $\qq = \mm$.

Observe that for each $n \in \NN$, there exists a unique homogeneous ideal $J'_n \subset A$ containing $\mathfrak{a}$ such that $J_n'/\mathfrak{a} \cong \ini_{\qq}(I_n)$ in $\gr_\qq(R)$. Let $\prec$ be the degree lexicographic monomial ordering in $A$, and as constructed in Section \ref{sec.notation}, let $J_n = \ini_{\prec}(J_n')$. It follows from Lemma \ref{L3}.(a) that $J_n$ is a monomial ideal in $A$ for all $n \ge 1$.

Observe that since $I_1$ and $\qq$ are $\mm$-primary ideals, there exists an integer $r \in \NN$ such that $\qq^r \subset I_1$. We then have
$$
\qq^rI_n \subset I_1I_n \subset I_{n+1},
$$
and
$$
\ini_{\qq}(\qq^i) \cong (\mm_A^i + \mathfrak{a})/\mathfrak{a}, \text{ for all } i \geq 1.
$$
This implies that
$$
\ini_{\qq}(\qq^r) \ini_{\qq}(I_n) \subset \ini_{\qq}(\qq^rI_n) \subset \ini_{\qq}(I_{n+1}).
$$
Thus,
$$
\mm_A^rJ'_n \subset (\mm_A^r + \mathfrak{a})J'_n + \mathfrak{a} \subset J'_{n+1}.
$$
Therefore,
$$
\mm_A^rJ_n = \ini_{\prec}(\mm_A^r)\ini_{\prec}(J_n') \subset \ini_{\prec}(\mm_A^rJ_n') \subset \ini_{\prec}(J_{n+1}') = J_{n+1}.
$$
Hence, by Proposition \ref{C1} and Lemma \ref{L3}.(c), we get
\begin{align} \label{E7}
\ell_R(R/I_{n+1}) - \ell_R(R/I_n) & = \ell_{R/\qq}(\gr_{\qq}(R)/\ini_{\qq}(I_{n+1})) - \ell_{R/\qq}(\gr_{\qq}(R)/\ini_{\qq}(I_n)) \\
& = \ell_{R/\qq}(A/J_{n+1}') - \ell_{R/\qq}(A/J_n') \notag \\
& = \ell_{R/\qq}(A/J_{n+1}) - \ell_{R/\qq}(A/J_n) \notag \\
& \leq \ell_{R/\qq}(A/\mm_A^rJ_n) - \ell_{R/\qq}(A/J_n). \notag
\end{align}

Note that $\qq^{rn} \subset I_1^n \subset I_n$, and so $\ini_{\qq}(\qq^{rn}) \subset \ini_{q}(I_n)$. This implies that
$$\mm_A^{rn} \subset \mm_A^{rn} + \mathfrak{a} \subset J'_n,$$
and therefore,
$$\mm_A^{rn} = \ini_{\prec}(\mm_A^{rn}) \subset \ini_{\prec}(J_n') = J_n.$$
It then follows from Proposition \ref{L2} that there exists a constant $\gamma_0 > 0$ such that
\begin{align} \label{E8}
\ell_{R/\qq}(A/\mm_A^rJ_n) - \ell_{R/\qq}(A/J_n) < \gamma_0 \cdot r^d \cdot (n+1)^{d-1}.
\end{align}
Putting (\ref{E7}) and (\ref{E8}) together, and choosing an appropriate constant $\gamma > 0$, we get the desired inequality. The theorem is proved.
\end{proof}

\begin{remark} A similar statement to Theorem \ref{T1}, replacing the length $\ell_R(R/I_n)$ by the multiplicity $e(I_n)$, was proved recently in \cite[Theorem 2.2]{C3}.
\end{remark}

\begin{example} \label{example1}
Let $R = k[[x_1, \dots, x_n]]$, for $d \ge 2$, and let $t \in \NN$ be a positive integer. Consider the filtration $\{I_n = \mm^{tn}\}_{n \in \NN}$. Then
$$\ell_R(R/I_n) = {d+tn-1 \choose d-1}.$$
It is then easy to see that
$$\ell_R(R/I_{n+1}) - \ell_R(R/I_n) > \frac{t^d}{(d-1)!}n^{d-1}.$$
Hence, the constant $\gamma$ satisfying the statement of Theorem \ref{T1} must be at least $t^d/(d-1)!$. This example shows that $\gamma$ in general may depend on the graded family $\{I_n\}_{n \in \NN}$.
\end{example}

\begin{example} \label{example}
In the proof of \cite[Theorem 5.5]{C2}, Cutkosky gave an example showing that when $d = 0$ the limit $\lim_{n \rightarrow \infty} \ell_R(R/I_n)$ does not necessarily exist. We shall recall his example to see that when $d=0$ the statement of Theorem \ref{T1} may fail. Let $(R,\mm)$ be an Artinian local ring for which the nilradical of the $\mm$-adic completion of $R$ is not 0. Then, there exists a constant $t$ such that $\mm^t \not= 0$ and $\mm^{t+1} = 0$. We inductively define a sequence $\{i_j\}_{j \ge 1}$ by setting $i_1 = 2$ and choosing $i_{j+1}$ to be an even number strictly larger than $2^ji_j$ for all $j \ge 1$. Define
$$\tau(n) = \left\{ \begin{array}{rcl} 0 & \text{if} & i_j \le n < i_{j+1} \text{ and } j \text{ is even} \\
1 & \text{if} & i_j \le n < i_{j+1} \text{ and } j \text{ is odd.} \end{array} \right.$$
Consider the graded family $\{I_n = \mm^{t + \tau(n)}\}_{n \in \NN}$. Observe that as $n \rightarrow \infty$, $\gamma n^{d-1} = \frac{\gamma}{n} \rightarrow 0$. Thus, in this case, the inequality
$$\ell_R(R/I_{n+1}) - \ell_R(R/I_n) < \gamma n^{d-1}$$
does not hold for $n \gg 0$.
\end{example}


\end{document}